\documentclass[reqno,a4paper,12pt]{amsart}

\usepackage[all,poly]{xy}
\usepackage{amsfonts,stmaryrd}
\usepackage[mathcal]{eucal}
\usepackage{amssymb}
\usepackage{amsmath}
\usepackage{mathrsfs}
\usepackage{color}
\usepackage[pagebackref,colorlinks]{hyperref}
\usepackage{enumerate}

\theoremstyle{plain}
\newtheorem{lemma}{Lemma}

\newtheorem {The}{Theorem}
\newtheorem{oldtheorem}{Theorem}

\newtheorem {Prob}{Problem}

\theoremstyle{remark}

\theoremstyle{definition}

\def\map{\longrightarrow}

\def\GL{\operatorname{GL}}

\def\K{\operatorname{K}}

\def\unlhd{\trianglelefteq}

\def\map{\longrightarrow}
\def\pamod#1{\,(\operatorname{mod}{\, #1})\,}

\def\map{\longrightarrow}

\def\epsilon{\varepsilon}

\def\map{\longrightarrow}

\def\K{\operatorname{K}}

\def\GL{\operatorname{GL}}

\long\def\forget#1\forgotten{}

\title[curiouser and curiouser]{Commutators of
elementary subgroups:\\ curiouser and curiouser}

\author{N.~Vavilov}
\address{Department of Mathematics and Computer Science,
St.~Petersburg State University, St.~Petersburg, Russia}
\email{nikolai-vavilov@yandex.ru}
\thanks{The work of the first author was supported by the
Russian Science Foundation grant 17-11-01261.}
\author{Z.~Zhang}
\address{Department of  Mathematics, Beijing Institute
of Technology, Beijing, China}
\email{zuhong@hotmail.com}


\begin{document}

\begin{abstract}
Let $R$ be any associative ring with $1$, $n\ge 3$, and let $A,B$ 
be two-sided ideals of $R$. In our previous joint works with
Roozbeh Hazrat \cite{Hazrat_Zhang_multiple,
Hazrat_Vavilov_Zhang} we have found a generating set for the
mixed commutator subgroup $[E(n,R,A),E(n,R,B)]$. Later in
\cite{NV18, NZ2} we noticed that our previous results can be 
drastically improved and that $[E(n,R,A),E(n,R,B)]$ is generated by  
1) the elementary conjugates $z_{ij}(ab,c)=t_{ij}(c)t_{ji}(ab)t_{ij}(-c)$ and $z_{ij}(ba,c)$, 2) the elementary commutators 
$[t_{ij}(a),t_{ji}(b)]$, where $1\le i\neq j\le n$, $a\in A$, $b\in B$, 
$c\in R$. Later in \cite{NZ1, NZ3} we noticed that for the second type 
of generators, it even suffices to fix one pair of indices $(i,j)$. Here
we improve the above result in yet another completely unexpected direction and prove that $[E(n,R,A),E(n,R,B)]$ is generated by the elementary commutators $[t_{ij}(a),t_{hk}(b)]$ alone, where 
$1\le i\neq j\le n$, $1\le h\neq k\le n$, $a\in A$, $b\in B$. This 
allows us to revise the technology of relative localisation, and, 
in particular, to give very short proofs for a number of recent results, 
such as the generation of partially relativised elementary groups 
$E(n,A)^{E(n,B)}$, 
multiple commutator formulas, commutator width, and the like.
\end{abstract}

\maketitle

In the present note we generalize the results by Roozbeh Hazrat 
and the authors \cite{Hazrat_Zhang_multiple, Hazrat_Vavilov_Zhang, NV18} on generation of mixed commutator subgroups of relative 
and unrelative elementary subgroups in the general linear group
in a completely unexpected direction.

Let $R$ be an associative ring with 1, and let $\GL(n, R)$ be the 
general linear group of degree $n\ge 3$ over $R$. As
usual, $e$ denotes the identity matrix, whereas $e_{ij}$ denotes a standard matrix unit. For $c\in R$ and $1\le i\neq j\le n$,
we denote by $t_{ij}(c) = e +ce_{ij}$ the corresponding {\it 
elementary transvection\/}. To an ideal $A\unlhd R$, we assign
the elementary subgroup
$$ E(n,A)=\big\langle t_{ij}(a),\ a\in A,\ 1\le i\neq j\le n\big\rangle. $$
\noindent
The corresponding relative elementary subgroup $E(n,R,A)$ is defined as the normal closure of $E(n,A)$ in the absolute elementary subgroup $E(n,R)$. From the work of Michael Stein, Jacques Tits, and Leonid
Vaserstein it is classically known that {\it as a group\/} $E(n,R,A)$
is generated by the {\it elementary conjugates\/}
$$ z_{ij}(a,c)=t_{ji}(c)t_{ij}(a)t_{ji}(-c), $$
\noindent
where $1\le i\neq j\le n$, $a\in A$, $c\in R$.


For $\GL(n,R)$ the study of the mixed commutator subgroups 
such as
$$ [\GL(n,R,A),\GL(n,R,B)],\quad [\GL(n,R,A),E(n,R,B)],\quad [E(n,R,A),E(n,R,B)] $$
\noindent
and other related birelative groups goes back to the 
ground-breaking work of Hyman Bass \cite{Bass_stable}, and
was then continued, again at the stable level by Alec Mason 
and Wilson Stothers \cite{Mason_Stothers, M2}, etc. For 
rings subject to commutativity conditions, this was then
resumed and expanded in several directions first by Hong You \cite{HongYou}, 
and then in our joint papers with 
Roozbeh Hazrat and Alexei Stepanov, see, for instance,
\cite{Vavilov_Stepanov_standard, Hazrat_Zhang, Vavilov_Stepanov_revisited, yoga-1, RNZ1, RNZ2, Hazrat_Zhang_multiple, yoga-2, Hazrat_Vavilov_Zhang_generation, RNZ5, Hazrat_Vavilov_Zhang}.
Those papers relied on the very powerful methods proposed
by Andrei Suslin, Zenon Borewicz and ourselves, Leonid Vaserstein, 
Tony Bak, and others to prove standard commutator formulas in 
the absolute case, see \cite{Suslin, borvav, Vaserstein_normal, Stepanov_Vavilov_decomposition, Bak}, etc., consult \cite{RN}
for a detailed survey.

A first version of following result was discovered (in a slightly
less precise form) by Roozbeh Hazrat and the second author, see \cite{Hazrat_Zhang_multiple}, Lemma 12. In exactly this form 
it is stated in our paper \cite{Hazrat_Vavilov_Zhang}, Theorem~3A.
The second type of generators below are the {\it elementary 
commutators\/}
$$ y_{ij}(a,b)=[t_{ij}(a),t_{ji}(b)] $$
\noindent
$1\le i\neq j\le n$, $a\in A$, $b\in B$. They belong already to 
the mixed commutator of the corresponding {\it unrelativised\/} subgroups $[E(n,A),E(n,B)]$.

\begin{oldtheorem}
Let $R$ be a quasi-finite ring with $1$, let $n\ge 3$, and let $A,B$ 
be two-sided ideals of $R$. Then the mixed commutator subgroup 
$[E(n,R,A),E(n,R,B)]$ is generated as a group by the elements of the form
\par\smallskip
$\bullet$ $z_{ij}(ab,c)$ and $z_{ij}(ba,c)$,
\par\smallskip
$\bullet$ $ y_{ij}(a,b)$,
\par\smallskip
$\bullet$ $[t_{ij}(a),z_{ij}(b,c)]$,
\par\smallskip\noindent
where $1\le i\neq j\le n$, $a\in A$, $b\in B$, $c\in R$.
\end{oldtheorem}

Here, the elementary conjugates generate
$E(n,R,A\circ B)$, where $A\circ B=AB+BA$ is the symmetrised
product of ideals.
Subsequently, we proved the following result, which
is both {\it terribly\/} much stronger, and much more general 
than Theorem~A and which completely solves \cite{Hazrat_Vavilov_Zhang}, Problem 1, for the case of $\GL_n$.
First, in \cite{NV18, NZ2} we noticed that the third type of
generators are redundant. Then in \cite{NV19, NZ1, NZ3} we observed 
that modulo $E(n,R,A\circ B)$ it suffices to take elementary 
commutators for a single position, and that everything works 
over {\it arbitrary\/} associative rings.

\begin{oldtheorem}
Let $R$ be any associative ring with $1$, let $n\ge 3$, and let $A,B$ 
be two-sided ideals of $R$. Then the mixed commutator subgroup 
$[E(n,R,A),E(n,R,B)]$ is generated as a group by the elements of the form
\par\smallskip
$\bullet$ $z_{ij}(ab,c)$ and $z_{ij}(ba,c)$,
\par\smallskip
$\bullet$ $y_{ij}(a,b)$,
\par\smallskip\noindent
where $1\le i\neq j\le n$, $a\in A$, $b\in B$, $c\in R$. 
Moreover, for the second type of generators, it suffices to fix 
one pair of indices $(i,j)$.
\end{oldtheorem}

In particular, this last result implies that for $n\ge 3$ one has
$$ [E(n,R,A),E(n,R,B)]=[E(n,A),E(n,B)] $$
\noindent
and in the body of the paper we usually prefer the shorter notation. 
\par
In the present paper we generalise both Theorem A, and the
first claim of Theorem B, in a completely different {\it unexpected\/}
direction. Namely, we prove that instead of limiting the stock
of elementary commutators, one can limit the stock of 
elementary conjugates! The following theorem asserts that
$[E(n,R,A),E(n,R,B)]$ is generated by the elementary commutators
not just over $E(n,R,A\circ B)$, as in our previous papers
\cite{NZ1, NZ2, NZ3}, but already over the unrelativised elementary subgroup $E(n,A\circ B)$.

\begin{The}\label{t1}
Let $A$ and $B$ be two ideals of an associative ring $R$ and let
$n\ge 3$. Then the mixed commutator subgroup 
$[E(n,R,A),E(n,R,B)]$ is generated by the elementary commutators
$[t_{ij}(a),t_{hk}(b)]$, where $1\le i\neq j\le n$, $1\le h\neq k\le n$,
$a\in A$ and $b\in B$.
\end{The}

We find this result truly astounding. In fact, it is well known that 
the relative elementary subgroups $E(n,R,A)$ themselves are not 
generated by the elementary transvections $t_{ij}(a)$, where 
$1\le i\neq j\le n$, $a\in A$. How come that their mixed commutator
subgroups are?

Actually, Theorem 1 immediately implies many remarkable 
corollaries. In particular, it allows to consider elementary 
commutators $y_{ij}(a,b)$, $a\in A$, $b\in B$, modulo the
{\it true\/} (= unrelative) elementary subgroup $E(n,A\circ B)$ 
of level $A\circ B$,
and not modulo the corresponding relative elementary subgroup 
$E(n,R,A\circ B)$, as we were doing in \cite{NZ1, NZ3, NZ4}.  

The balance of the paper is organised as follows.
In \S~1 we recall some basic facts concerning elementary 
subgroups and their mixed commutator subgroups. In \S~2
we prove Lemma~6 that allows to move around elementary commutators. It is essentially a slight improvement of
our results from \cite{NZ1, NZ3, NZ4} that is not directly needed 
to prove Theorem 1 itself, but serves as a model, and is crucial 
to establish Theorem 4. 
After that in \S~3 we prove Lemma 7 which is essentially a 
slightly more precise form both of level computations and of 
Theorem 1. The rest of the paper are refinements and applications. 
\par\smallskip
$\bullet$ In \S~4 we derive a generation result 
for partially relativised elementary groups 
$E(n,B,A)=E(n,A)^{E(n,B)}$, where $A$ and $B$ are two ideals, Theorem 2. 
\par\smallskip
$\bullet$ In \S~5 prove that already the true elementary
subgroup $E(n,A\circ B)$ is normal in $[E(n,A),E(n,B)]$, with 
abelian quotient $[E(n,A),E(n,B)]/E(n,A\circ B)$, Theorem 3.
\par\smallskip
$\bullet$ In \S~6 we notice that Theorem 1 admits a remarkable 
generalisation. Namely the generating set therein can be further 
substantially reduced, allowing only elementary commutators 
corresponding to some positions in the unipotent radical of a 
maximal parabolic subgroup plus the elementary commutators
in {\it one more\/} position, Theorem~4.
\par\smallskip
$\bullet$ In \S~7 we apply the above results to the generation 
of mixed commutator subgroups. Modulo the results of our
previous paper \cite{NZ3}, we are now in a position to exhibit
very economical generating sets for the multiple commutator
subgroups $\llbracket E(n,I_1),\ldots,E(n,I_m)\rrbracket$,
Theorem 5. 
\par\smallskip\noindent
Finally, in \S~8 we briefly mention some further applications
and unsolved problems.


\section{Relative subgroups}

Let $G=\GL(n,R)$ be the general linear group of degree $n$ over an
associative ring $R$ with 1. In the sequel for a matrix
$g\in G$ we denote by $g_{ij}$ its matrix entry in the position
$(i,j)$, so that $g=(g_{ij})$, $1\le i,j\le n$. The inverse of
$g$ will be denoted by $g^{-1}=(g'_{ij})$, $1\le i,j\le n$.
\par
As usual we denote by $e$ the identity matrix of degree $n$
and by $e_{ij}$ a standard matrix unit, i.~e., the matrix
that has 1 in the position $(i,j)$ and zeros elsewhere.
An elementary transvection $t_{ij}(\xi)$ is a matrix of the
form $t_{ij}(c)=e+c e_{ij}$, $1\le i\neq j\le n$, $c\in R$. 
\par
Further, let $A$ be a two-sided of $R$. We consider the
corresponding reduction homomorphism
$$ \pi_A:\GL(n,R)\map\GL(n,R/A),\quad
(g_{ij})\mapsto(g_{ij}+A). $$
\noindent
Now, the {\it principal congruence subgroup} $\GL(n,R,A)$ of level $A$
is the kernel $\pi_A$, 
\par
The {\it unrelative elementary subgroup} $E(n,A)$ of level $A$
in $\GL(n,R)$ is generated by all elementary matrices of level 
$A$. In other words,
$$ E(n,A)=\langle e_{ij}(a),\ 1\le i\neq j\le n,\ a\in A \rangle. $$
\noindent
In general $E(n,A)$ has little chances to be normal in $\GL(n,R)$. 
The {\it relative elementary subgroup} $E(n,R,A)$  of level $A$
is defined as the normal closure of $E(n,A)$ in the absolute 
elementary subgroup $E(n,R)$:
$$ E(n,R,A)=\langle e_{ij}(a),\ 1\le i\neq j\le n,\ a\in A 
\rangle^{E(n,R)}. $$

The following lemma in generation of relative elementary subgroups
$E(n,R,A)$ is a classical result discovered in various contexts by 
Stein, Tits and Vaserstein, see, for instance, \cite{Vaserstein_normal}
(or \cite{Hazrat_Vavilov_Zhang}, Lemma 3, for a complete elementary
proof). It is stated in terms of the {\it Stein---Tits---Vaserstein 
generators\/}):
$$ z_{ij}(a,c)=t_{ij}(c)t_{ji}(a)t_{ij}(-c),\qquad
1\le i\neq j\le n,\quad a\in A,\quad c\in R. $$

\begin{lemma}
Let $R$ be an associative ring with $1$, $n\ge 3$, and let $A$ 
be a two-sided ideal of $R$. Then
$$ E(n,R,A) = \big\langle z_{ij}(a,c),\ 1\le i\neq j\le n, 
a\in A, c\in R \big\rangle. $$
\end{lemma}

We also need the following results on the mixed 
commutators $[E(n,A),E(n,B)]$.
Everywhere below the commutators are left-normed 
so that for two elements $x,y$ of a group $G$ one has 
$[x,y]={}^{x}y\cdot y^{-1}=x\cdot{}^y(x^{-1})=xyx^{-1}y^{-1}$,
where ${}^xy=xyx^{-1}$ is the left conjugate of $y$ by $x$. In the sequel
we repeatedly use standard commutator identities such as 
$[xy,z]={}^x[y,z]\cdot [x,z]$ or $[x,yz]=[x,y]\cdot{}^y[x,z]$ 
without any explicit reference.

Denote by $A\circ B=AB+BA$ the symmetrised product of two-sided ideals $A$ and $B$. For commutative rings, $A\circ B=AB=BA$
is the usual product of ideals $A$ and $B$. However, in general, 
the symmetrised product is not associative. Thus, when writing something like $A\circ B\circ C$, we have to specify the order 
in which products are formed. The following level computation is
standard, see, for instance, \cite{Vavilov_Stepanov_standard,
Vavilov_Stepanov_revisited, Hazrat_Vavilov_Zhang}, and references
there.

\begin{lemma}
$R$ be an associative ring with $1$, $n\ge 3$, and let $A$ and $B$
be two-sided ideals of $R$.  Then 
$$ E(n,R,A\circ B)\le\big[E(n,A),E(n,B)\big]\le
\big[E(n,R,A),E(n,R,B)\big] \le\GL(n,R,A\circ B). $$
\end{lemma}

Since all generators listed in Theorem~B belong already to the
commutator subgroup of unrelative elementary subgroups, we
get the following corollary, \cite{NZ2}, Theorem 2.

\begin{lemma}
Let $R$ be any associative ring with $1$, let $n\ge 3$, and let $A,B$ 
be two-sided ideals of $R$.  Then one has
$$ \big[E(n,R,A),E(n,R,B)\big]=\big[E(n,R,A),E(n,B)\big]=\big[E(n,A),E(n,B)\big]. $$
\end{lemma}

In particular, it follows that $[E(n,A),E(n,B)]$ is normal in
$E(n,R)$.
Our proof of Theorem 1 is based on the following computation,
which is contained within the proof of \cite{NZ1}, Lemma 3, or 
\cite{NZ3}, Lemma 9.

\begin{lemma}
Let $R$ be an associative ring with $1$, $n\ge 3$. For any 
three pair-wise distinct indices $i,j,h$ and any $a,b,c\in R$
one has
\begin{align*}
&[t_{ih}(c),y_{ij}(a,b)]=t_{ih}(-abc-ababc)t_{jh}(-babc),\\
&[t_{jh}(c),y_{ij}(a,b)]=t_{ih}(abac)t_{jh}(bac),\\
&[t_{hi}(c),y_{ij}(a,b)]=t_{hi}(cab)t_{hj}(-caba),\\
&[t_{hj}(c),y_{ij}(a,b)]=t_{hi}(cbab)t_{hj}(-cba-cbaba).
\end{align*}
\end{lemma}
When $a\in A$ and $b\in B$ it follows that 
$$ {}^{t_{kl}(c)}y_{ij}(a,b)\equiv y_{ij}(a,b)
\pamod{E(n,A\circ B)}, $$
\noindent
unless $(k,l)=(i,j),(j,i)$.
 

\section{Rolling over elementary commutators}

Our proof is yet another variation on the following theme.
The result was observed by Wilberd van der Kallen, as part 
of the proof 
of \cite{vdK-group}, Lemma 2.2, that modulo the elementary transvections of a given level elementary conjugates are 
connected by a triple identity that allows to express one of
them in terms of two other ones in different positions. It is 
reproduced in this precise form with a detailed proof in 
\cite{NZ4}, Lemma 12.

\begin{lemma}\label{l11}
Let $A\unlhd R$ be an ideal of an associative ring, $n\ge 3$, and
let $i,j,h$ be three pair-wise distinct indices. Then if a subgroup
$E(n,A)\le H\le\GL(n,R)$ contains 
\par\smallskip
$\bullet$ either $z_{ih}(a,c)$ and $z_{jh}(a,c)$,
\par\smallskip
$\bullet$ or $z_{hi}(a,c)$ and $z_{hj}(a,c)$,
\par\smallskip\noindent
for all $a\in A$, $c\in R$, then it also contains $z_{ij}(a,c)$ 
and $z_{ji}(a,c)$, for all such $a$ and $c$.
\end{lemma}

In this section, and the next one, we establish its counterparts for elementary commutators in
$[E(n,A),E(n,B)]$.
The following result is {\it essentially\/} \cite{NZ1}, Lemma~5
or \cite{NZ3}, Lemma 11. Of course, there it was stated in a 
weaker form, as a congruence between elementary commutators modulo the relative elementary subgroup
$E(n,R,A\circ B)$, without specifying that we actually only need
elementary conjugates in {\it one\/} position. As a result, the
calculations in \cite{NZ1, NZ3} were not residing at the level
in elementaries, at the moment a factor from $E(n,R,A\circ B)$
occurred, it was immediately discarded. Here, we have to come 
up with a genuine calculation of all terms.

\begin{lemma}\label{l13}
Let $A, B\unlhd R$ be two-sided ideals of an associative ring, 
$n\ge 3$, and
let $i,j,h$ be three pair-wise distinct indices. Then if a subgroup
$$ E(n,A\circ B)\le H\le\GL(n,R) $$ 
\noindent
contains 
\par\smallskip
$\bullet$ either $y_{ih}(a,b)$ and $z_{jh}(ba,c)$,
\par\smallskip
$\bullet$ or $z_{hi}(ab,c)$ and $y_{hj}(a,b)$,
\par\smallskip\noindent
for all $a\in A$, $b\in B$, then it also contains $y_{ij}(a,b)$, 
for all such $a$ and $b$ and all $c\in R$.
\end{lemma}

\begin{proof}
Take any $ h\neq i,j$ and rewrite the elementary commutator
$$ z=y_{ij}(a,b)=\big[t_{ij}(a),t_{ji}(b)\big] $$ 
\noindent
in the following form
$$ z=t_{ij}(a)\cdot{}^{t_{ji}(b)}t_{ij}(-a)=
t_{ij}(a)\cdot{}^{t_{ji}(b)}\big[t_{ih}(a),t_{hj}(-1)\big]. $$
\noindent
Expanding the conjugation by $t_{ji}(b)$, we see that 
$$ z=
t_{ij}(a)\cdot\big[{}^{t_{ji}(b)}t_{ih}(a),{}^{t_{ji}(b)}t_{hj}(-1)\big] = 
t_{ij}(a)\cdot\big[t_{jh}(ba)t_{ih}(a),t_{hj}(-1)t_{hi}(b)\big]. $$
\noindent
Using multiplicativity of the commutator w.r.t. the first argument 
we get
$$ z = t_{ij}(a)\cdot 
{}^{t_{jh}(ba)}\big[t_{ih}(a),t_{hj}(-1)t_{hi}(b)\big]\cdot
\big[t_{jh}(ba),t_{hj}(-1)t_{hi}(b)\big]. $$
\noindent
Using multiplicativity of the commutator w.r.t. the second argument,
\begin{multline*}
z = t_{ij}(a)\cdot 
{}^{t_{jh}(ba)}[t_{ih}(a),t_{hj}(-1)]\cdot 
{}^{t_{jh}(ba)t_{hj}(-1)}[t_{ih}(a),t_{hi}(b)]\cdot\\
[t_{jh}(ba),t_{hj}(-1)]\cdot 
{}^{t_{hj}(-1)}[t_{jh}(ba),t{hi}(b)] .
\end{multline*}
\noindent
Let us look at the factors separately.
\par\medskip
$\bullet$ 
The product of the first two factors equals 
$t_{ih}(aba)\in E(n,A\circ B)\le H$.
\par\medskip
$\bullet$ In the third factor one has $t_{ji}(ba)\in E(n,A\circ B)$, 
so that the corresponding conjugation can be discarded, whereas
$$ {}^{t_{hj}(-1)}y_{ih}(a,b)\equiv y_{ih}(a,b)
\pamod{E(n,A\circ B)} $$
\noindent
belongs to $H$.
\par\medskip
$\bullet$ The fourth factor equals $t_{jh}(ba)z_{jh}(-ba,-1)\in H$.
\par\medskip
$\bullet$ Finally, the last factor equals
$t_{ji}(bab)t_{hi}(-bab)\in E(n,A\circ B)\le H$.
\end{proof}


\section{Proof of Theorem 1}

But then, the calculation in the preceding section can be reversed, 
to express elementary conjugates as products of elementary 
commutators. This is accomplished in the following lemma. 
Again, {\it essentially\/} this lemma is based on the same calculation
that was used in level calculations to prove the leftmost inclusion in
Lemma 2, see, for instance, \cite{Hazrat_Vavilov_Zhang}, Lemma 1A. But it was never stated in this precise form.
This lemma immediately implies Theorem 1, but its full force will
be revealed in \S~6, where it will be used to establish much
more precise results.

\begin{lemma}\label{l12}
Let $A, B\unlhd R$ be two-sided ideals of an associative ring, 
$n\ge 3$, and
let $i,j,h$ be three pair-wise distinct indices. Then if a subgroup
$$ E(n,A\circ B)\le H\le\GL(n,R) $$ 
\noindent
contains 
\par\smallskip
$\bullet$ either $y_{ih}(a,b)$ and $y_{jh}(a,b)$,
\par\smallskip
$\bullet$ or $y_{hi}(b,a)$ and $y_{hj}(b,a)$,
\par\smallskip\noindent
for all $a\in A$, $b\in B$, then it also contains $z_{ij}(ab,c)$ 
for all such $a$ and $b$, and all $c\in R$.
Similarly for the elementary conjugates $z_{ij}(ba,c)$ with
$A$ and $B$ interchanged.
\end{lemma}

\begin{proof}
Since the condition is symmetric with respect to $A$ and $B$,
and $y_{ij}(b,a)^{-1}=y_{ji}(a,b)\in H$ implies that 
$y_{ij}(b,a)\in H$, it suffices to consider one of the four 
occurring cases. Thus, we only have to verify that 
$z_{ij}(ab,c)\in H$ provided that $y_{ih}(a,b),y_{jh}(a,b)\in H$.
\par
Obviously, $t_{ij}(ab)=[t_{ih}(a),t_{hj}(b)]$. Decompose $z=z_{ij}(ab,c)$ accordingly:
$$ z=z_{ij}(ab,c)={}^{t_{ji}(c)}t_{ij}(ab)=
{}^{t_{ji}(c)}[t_{ih}(a),t_{hj}(b)]=
[t_{jh}(ca)t_{ih}(a),t_{hj}(b)t_{hi}(-bc)]. $$
\noindent 
Using multiplicativity of commutators w.r.t. the first argument
we see that
$$ z={}^{t_{jh}(ca)}[t_{ih}(a),t_{hj}(b)t_{hi}(-bc)]\cdot
[t_{jh}(ca), t_{hj}(b)t_{hi}(-bc)]. $$
\noindent
Expanding both factors in the above expression for $z$
w.r.t. the second argument, we get
\begin{multline*}
z={}^{t_{jh}(ca)}[t_{ih}(a),t_{hj}(b)]\cdot
{}^{t_{jh}(ca)t_{hj}(b)}[t_{ih}(a),t_{hi}(-bc)]\cdot \\
[t_{jh}(ca), t_{hj}(b)]\cdot
{}^{t_{hj}(b)}[t_{jh}(ca),t_{hi}(-bc)].
\end{multline*}
\noindent
Consider the four factors separately. Clearly,
\par\medskip
$\bullet$ ${}^{t_{jh}(ca)}[t_{ih}(a),t_{hj}(b)]=
{}^{t_{jh}(ca)}t_{ij}(ab)=t_{ij}(ab)t_{ih}(-abca)\in H$.
\par\medskip
$\bullet$ $[t_{jh}(ca),t_{hj}(b)]=y_{jh}(ca,b)\in H$.
\par\medskip
$\bullet$ ${}^{t_{hj}(b)}[t_{jh}(ca),t_{hi}(-bc)]=
{}^{t_{hj}(b)}t_{ji}(-cabc)=t_{hi}(-bcabc)t_{ji}(-cabc)\in H$.
\par\medskip\noindent
Thus, the only slightly problematic factor is the second one,
$$ w={}^{t_{jh}(ca)t_{hj}(b)}[t_{ih}(a),t_{hi}(-bc)]. $$
\par\medskip
$\bullet$ However, 
$$ w={}^{t_{jh}(ca)t_{hj}(b)}y_{ih}(a,-bc)=
{}^{t_{jh}(ca)}\Big([t_{hj}(b),y_{ih}(a,-bc)]\cdot y_{ih}(a,-bc)\Big). $$
\noindent
By Lemma 4 one has
$$ [t_{hj}(b),y_{ih}(a,-bc)]=t_{ij}(-abcab)t_{hj}(-bcac). $$
\noindent
Plugging this into the above expression for $w$ we get
\begin{multline*}
w={}^{t_{jh}(ca)}\Big(t_{ij}(-abcab)\cdot t_{hj}(-bcac)\cdot y_{ih}(a,-bc)\Big)= \\
{}^{t_{jh}(ca)}t_{ij}(-abcab)\cdot
{}^{t_{jh}(ca)}t_{hj}(-bcac)\cdot
{}^{t_{jh}(ca)}y_{ih}(a,-bc). 
\end{multline*}
\noindent
Clearly, the first two factors
\begin{align*}
&{}^{t_{jh}(ca)}t_{ij}(-abcab)=t_{ij}(-abcab)t_{ih}(abcabca),\\
\noalign{\vskip 3truept}
&{}^{t_{jh}(ca)}t_{hj}(-bcac)=[t_{jh}(ca),t_{hj}(-bcac)]\cdot
t_{hj}(-bcac)=y_{jh}(ca,-bcac)\cdot t_{hj}(-bcac),
\end{align*}
\noindent
both belong to $H$. On the other hand, using Lemma 4 once more
we get
\begin{multline*}
{}^{t_{jh}(ca)}y_{ih}(a,-bc)=
[t_{jh}(ca),y_{ih}(a,-bc)]\cdot y_{ih}(a,-bc)=\\
t_{ji}(cabcabc)\cdot t_{jh}(cabca-cabcabca)\cdot y_{ih}(a,-bc)\in H. 
\end{multline*}
\par
This means that the second factor in the above expression of
$z$ also belongs to $H$ and we are done.
\end{proof}


\section{Partially relativised subgroups}

For two ideals $A,B\unlhd R$ we denote by $E(n,B,A)$ 
the {\it partially relativised elementary group\/}, which is the 
smallest subgroup containing $E(n,A)$ and normalised by $E(n,B)$:
$$ E(n,B,A)=E(n,A)^{E(n,B)}. $$
\noindent
In particular, when $B=R$ we get the usual relative group
$E(n,R,A)$, as defined above. 
\par\smallskip\noindent
{\bf Remark.\ } In a special case the groups 
$E(n,B,A)$ were first systematically considered by Roozbeh 
Hazrat and the second author in their works on relative localisation
\cite{Hazrat_Zhang, Hazrat_Zhang_multiple}. It soon became clear
that  $E(n,R,s^mI)$ were too large to serve as a convenient
system of neighborhoods of 1, whereas $E(n,s^mI)$ were way
too small. The partially relativised groups $E(n,s^mR,s^mI)$ 
turned out to be a
much smarter choice. This important technical innovation then 
proved extremely useful in our joint
works with Roozbeh Hazrat, see, in particular, 
\cite{RNZ1, RNZ2, RNZ5}. Later, partially relativised elementary
groups figured prominently in the universal localisation by Alexei Stepanov \cite{Stepanov_nonabelian, Stepanov_universal, AS}.
\par\smallskip
The following obvious observation relates partially relativised
subgroups with double commutators of elementary subgroups.

\begin{lemma}\label{l7}
Let\/ $R$ be any associative ring with\/ $1$, and let\/ $A,B$ 
be two-sided ideals of\/ $R$. Then for any $n\ge 2$ one has
$$ E(n,B,A)=[E(n,A),E(n,B)]\cdot E(n,A). $$
\end{lemma}

\begin{proof}
By the very definition $E(n,A)\le E(n,B,A)$. The mixed 
commutator subgroup $[E(n,A),E(n,B)]$ is generated 
by the commutators
$[x,y]=x(yx^{-1}y^{-1})$, where $x\in E(n,A)$, 
$y\in E(n,B)$. Thus, $[E(n,A),E(n,B)]\le E(n,B,A)$. This means
that the left hand side is contained in the right hand side.
\par
Conversely, observe that the product on the right hand side is
a subgroup. Indeed, ${}^x[y,z]=[xy,z]\cdot{[x,z]}^{-1}$, 
where $x,y\in E(n,A)$, $z\in E(n,B)$. It follows that 
$E(n,A)$ normalises $[E(n,A),E(n,B)]$. Finally, for all
such $x$ and $z$ one has
${}^zx=[z,x]\cdot x$, so that the right hand side 
contains all generators of $E(n,B,A)$.
\end{proof}

The following result is \cite{NZ4}, Theorem 2. There it was derived 
from Theorem B by another calculation in the style of level 
calculations, now, with Theorem 1 on deck, it becomes immediate.

\begin{The}\label{t2}
Let $R$ be an associative ring with identity $1$, $n\ge 3$,
and let $A$ and $B$ be two-sided ideals of $R$. Then
$$ E(n,B,A) = \big\langle z_{ij}(a,b),\ 1\le i\neq j\le n, 
a\in A, b\in B \big\rangle. $$
\end{The}

\begin{proof}
Consider the subgroup 
$$ H=\big\langle z_{ij}(a,b),\
1\le i\neq j\le n, a\in A, b\in B\big\rangle, $$
\noindent
generated by the elementary conjugates contained in
$E(\Phi,B,A)$. By the very definition, $z_{ij}(a,b)\in E(n,B,A)$
so that $H\le E(n,B,A)$.
\par
To prove the converse inclusion, it suffices to verify that 
the generators of $E(n,B,A)$ listed in Lemma 7 are in 
fact contained already in $H$.
\par\smallskip
$\bullet$ By definition, any $x\in E(n,A)$ is a product of 
the elementary generators $x_{ij}(a)=z_{ij}(a,0)$. In
other words, $E(n,A)\le H$.
\par\smallskip
$\bullet$ By Theorem~\ref{t1} modulo $E(n,A\circ B)\le E(n,A)$
the mixed commutator subgroup $[E(n,A),E(n,B)]$ is generated 
by the elementary commutators $[t_{ij}(a),t_{hk}(b)]$, where 
$1\le i\neq j\le n$, $1\le h\neq k\le n$, $a\in A$ and $b\in B$.
However, 
$$ y_{ij}(a,b)=[x_{ij}(a),x_{ji}(b)]=x_{ij}(a)\cdot 
{}^{x_{ji}(b)}x_{ij}(-a)=z_{ij}(a,0)z_{ij}(-a,b)\in H, $$
\noindent
which finishes the proof. 
\end{proof}


\section{Elementary commutators modulo $E(n,A\circ B)$}

In \cite{NZ1, NZ3} Lemma 4 was used to establish that 
the quotient
$$ \big[E(n,R,A),E(n,R,B)\big]/E(n,R,A\circ B)=
\big[E(n,A),E(n,B)\big]/E(n,R,A\circ B) $$ 
\noindent
is central in $E(n,R)/E(n,R,A\circ B)$. 

\begin{lemma}
Let $R$ be an associative ring with $1$, $n\ge 3$, and let $A,B$ 
be two-sided ideals of $R$. Then
$$ \big[\big[E(n,A),E(n,B)\big],E(n,R)\big]=E(n,R,A\circ B). $$
\end{lemma}

In particular, $\big[E(n,A),E(n,B)\big]/E(n,R,A\circ B)$ is
itself abelian. In turn, Lemma~9 was itself a key step in the proof 
of the following more general result on triple commutators.
The following result is \cite{NZ3}, Lemma 7.

\begin{lemma}
Let $R$ be an associative ring with $1$, $n\ge 3$, and let $A,B,C$ 
be three two-sided ideals of $R$. Then
$$ \big[\big[E(n,A),E(n,B)\big],E(n,C)\big]=
\big[E(n,A\circ B),E(n,C)\big]. $$
\end{lemma}

In \cite{NZ3} we succeeded in proving an analogue of Lemma 10 for
quadruple commutators only under the stronger assumption
that $n\ge 4$, see \cite{NZ3}, Lemma 8. However, it follows
from the standard commutator formula that over {\it quasi-finite\/}
rings the claim holds also for $n=3$, see 
\cite{Hazrat_Zhang_multiple, Hazrat_Vavilov_Zhang}.

\begin{lemma}
Assume that either $R$ is an arbitrary associative ring with $1$ 
and $n\ge 4$, or $n=3$ and $R$ is quasi-finite. Further, let $A,B,C,D$ 
be four two-sided ideals of $R$. Then
$$ \big[\big[E(n,A),E(n,B)\big],\big[E(n,C),E(n,D)\big]\big]=
\big[E(n,A\circ B),E(n,C\circ D)\big]. $$
\end{lemma}

When writing \cite{NZ1, NZ3} and even \cite{NZ4} 
we failed to notice the following generalisation of Lemma 9.

\begin{The}\label{t3}
Let $R$ be an associative ring with identity $1$, $n\ge 3$,
and let $A$ and $B$ be two-sided ideals of $R$. Then
$$ E(n,A\circ B)\unlhd \big[E(n,A),E(n,B)\big] $$
\noindent
and the quotient $E(n,R,A\circ B)/E(n,A\circ B)$ is
central in 
$$ \big[E(n,A),E(n,B)\big]/E(n,A\circ B). $$
\noindent
If, moreover, either $n\ge 4$, or $R$ is quasi-finite, then
this last quotient is itself abelian. 
\end{The}
\begin{proof}
By Lemma 10 one has
$$ \big[\big[E(n,A),E(n,B)\big],E(n,R,A\circ B)\big]
=\big[E(n,A\circ B),E(n,A\circ B)\big]
\le E(n,A\circ B), $$
\noindent
which establishes the first two claims of lemma. To check the
last claim, observe that under these additional assumptions
one has
$$ \big[\big[E(n,A),E(n,B)\big],\big[E(n,A),E(n,B)\big]\big]
=\big[E(n,A\circ B),E(n,A\circ B)\big]
\le E(n,A\circ B), $$
\noindent
now by Lemma 11.
\end{proof}


This means that instead of studying the quotient 
$[E(n,A),E(n,B)]/E(n,R,A\circ B)$, as we did in \cite{NZ1,
NZ3, NZ4}, we can now study the larger quotient
$$ [E(n,A),E(n,B)\big]/E(n,A\circ B). $$
\noindent 
In \cite{NZ1, NZ3, NZ4} we retrieved some of the relations 
among the elementary commutators $y_{ij}(a,b)$ modulo 
the relative elementary subgroup $E(n,R,A\circ B)$. However, 
a naive attempt to generalise these relations by eliminating the elementary 
conjugates at the cost of introducing further elementary
commutators leads to complicated and unsavoury relations.


\section{Further reducing the generating set of $[E(n,A),E(n,B)]$}

Let us state a result which is {\it essentially\/} due to Wilberd van 
der Kallen and Alexei Stepanov. Namely, in \cite{vdK-group}, 
Lemma 2.2 this result is established for the unipotent radical
of a {\it terminal\/} parabolic in $\GL(n,R)$. Subsequently, it
was generalised to all maximal parabolics in Chevalley groups
in \cite{Stepanov_calculus, Stepanov_nonabelian, 
Stepanov_universal}, but of course, in these papers $R$ was 
always assumed to be commutative. In that form, it is stated 
as corollary to \cite{NZ3}, Theorem 3 (of course, it immediately 
follows already from Lemma 5 above = \cite{NZ3}, Lemma 12).
Morally, it is a trickier and mightier version of the classical 
generation result for $E(n,R,A)$, Lemma 1, with a smaller 
set of generators. Actually, one does not even need the whole 
unipotent radical, just $n-1$ roots in interlaced positions that 
allow to repeatedly apply Lemma 5.

\begin{oldtheorem}
Let $A\unlhd R$ be a two-sided ideal of an associative ring, 
$n\ge 3$ and let $1\le r\le n-1$. Then the relative elementary 
subgroup $E(n,R,A)$ is generated by the true elementary 
subgroup $E(n,A)$ and the elementary conjugates 
$z_{ij}(a,c)$, where $a\in A$, $c\in R$, whereas $(i,j)$
is one of the following\/{\rm:}
\par\smallskip
$\bullet$ Either $(i,h)$, for all $1\le i\le r$ and a fixed
$r+1\le h\le n$,
\par\smallskip
$\bullet$ or $(k,j)$, for a fixed $1\le k\le r$ and all 
 $r+1\le j\le n$.
\end{oldtheorem}

\begin{proof}
Denote the subgroup generated by $E(n,A)$ and the above
elementary conjugates by $H$. By Lemma 1 it suffices to
verify that $z_{ij}(a,c)\in H$ for all $1\le i\neq j\le n$.
\par\smallskip
$\bullet$ First, let $r=1$ or $r=n-1$. In this case our theorem is precisely \cite{vdK-group}, Lemma 2.2. Indeed, let, say, $r=n-1$. 
Then $H$ contains $z_{in}(a,c)$, for all $1\le i\le n-1$, and thus
by Lemma~5 one has $z_{ij}(a,c)\in H$ for all $1\le i\neq j\le n-1$.
Now, fix an $h$, $1\le h\le n-1$. Applying Lemma~5 to 
$z_{hj}(a,c)\in H$, $j\neq h$, we can conclude that 
$z_{nj}(a,c)\in H$, $j\neq h$. Since $h$ here is arbitrary and 
$n-1\ge 2$, it follows that $z_{nj}(a,c)\in H$ for all $1\le j\le n-1$.
\par\smallskip
$\bullet$
Now, let $2\le r\le n-2$. Without loss of generality, we can assume
that $h=r+1$ and $k=1$. Applying the previous case to the group
$\GL(r+1,R)$ embedded in $\GL(n,R)$ to the first $r+1$ rows and
columns, we get that $z_{ij}(a,b)\in H$ for all $1\le i\neq j\le r+1$.
In particular, $z_{1j}(a,c)\in H$ for all $j\neq 1$. Applying the
previous case again, now to the group $\GL(n,R)$ itself, we
can conclude that $z_{ij}(a,c)\in H$ for all $1\le i\neq j\le n$,
as claimed.
\end{proof}

Combined with Theorem B it immediately
yields the following further sharpening of Theorems A and B.
Notice, that $[E(n,A),E(n,B)]$ is in general {\it strictly larger\/}
than $E(n,R,A\circ B)$, see, for instance, the discussion in 
\cite{NZ3}, \S~7. Thus, we need an extra type of generators, 
for one more position, which gives us $n$ positions.

\begin{oldtheorem}
Let $R$ be any associative ring with $1$, let $n\ge 3$ and let 
$1\le r\le n-1$.  Further, let $A,B$ be two-sided ideals of $R$. 
Then the mixed commutator subgroup $[E(n,A),E(n,B)]$ is 
generated as a group by the true elementary subgroup 
$E(n,A\circ B)$ and the elements of the two following forms\/{\rm:}
\par\smallskip
$\bullet$ $z_{ij}(ab,c)$ and $z_{ij}(ba,c)$,
\par\smallskip
$\bullet$ $y_{st}(a,b)$,
\par\smallskip\noindent
where $a\in A$, $b\in B$, $c\in R$, the pairs of indices $(i,j)$ 
are as in Theorem~{\rm C}, while $(s,t)$, $1\le s\neq t\le n$, is 
an arbitrary fixed pair of indices.
\end{oldtheorem}

Recall that Theorem B itself was based on a version of Lemma 6,
while the derivation of Theorem D from Theorem B ultimately
depends on Lemma 5.
On the other hand, now that we have Lemma 7, expressing
elementary conjugates in terms of elementary commutators,
we can prove
a similar sharper form of Theorem 1. It is another main result of 
the present paper and a counterpart of Theorems C and D for 
mixed commutator subgroups of elementary groups.

\begin{The}\label{tC}
Let $A,B\unlhd R$ be two-sided ideal of an associative ring, 
$n\ge 3$, and let $1\le r\le n-1$. Then the mixed commutator 
$[E(n,A),E(n,B)]$ is generated by:
\par\smallskip
$\bullet$ The true elementary subgroup $E(n,A\circ B)$.
\par\smallskip
$\bullet$ The elementary commutators $y_{ij}(a,b)$, 
for all $a\in A$, $b\in B$, the pairs of indices $(i,j)$ 
are as in Theorem~{\rm C},
and the elementary conjugates/commutators for \emph{one more\/}
position.
\par\smallskip
$\bullet$ Either the elementary conjugates $z_{st}(ab,c)$ and 
$z_{st}(ba,c)$ for all
$a\in A$, $b\in B$, $c\in R$ and any position $(s,t)$ 
such that either $1\le s\le r$ and $r+1\le t\le n$, or 
$1\le t\le r$ and $r+1\le s\le n$.
\par\smallskip
$\bullet$ Or the elementary commutators $y_{st}(a,b)$, 
for all $a\in A$, $b\in B$, and 
any position $(s,t)$ such that either $1\le s\neq t\le r$, or
$r+1\le s\neq t\le n$.
\end{The}

\begin{proof}
Denote the subgroup generated by $E(n,A\circ B)$ and the 
elementary commutators $y_{ij}(a,b)$ by $H$. First, we
verify that in that case $z_{ij}(ab,c),z_{ij}(ba,c)\in H$ for all 
$1\le i\neq j\le r$ or $r+1\le i\neq j\le n$.
\par\smallskip
$\bullet$ As above, first we treat the case, where $r=1$ or 
$r=n-1$. Let, say, $r=n-1$. Then $H$ contains $y_{in}(a,b)$, for 
all $1\le i\le n-1$, and thus by Lemma~7 one has $z_{ij}(ab,c)\in H$ 
and $z_{ij}(ba,c)\in H$ for all $1\le i\neq j\le n-1$. As above, fix 
an $h$, $1\le h\le n-1$. Applying Lemma~6 to 
$z_{hj}(ab,c)\in H$, $j\neq h,n$ and $y_{hn}(a,b)$, we can 
conclude that $y_{nj}(a,b)\in H$, $j\neq h$. Since $h$ here 
is arbitrary and $n-1\ge 2$, it follows that $y_{nj}(a,b)\in H$ 
for all $1\le j\le n-1$.
\par\smallskip
$\bullet$
Now, let $2\le r\le n-2$ and let $h$ and $k$ be as in the statement
of Theorem C. Without loss of generality, we can assume
that $h=r+1$ and $k=r$. Applying the previous case to the group
$\GL(r+1,R)$ embedded in $\GL(n,R)$ on the first $r+1$ rows and
columns, we get that $z_{ij}(ab,c),z_{ij}(ba,c)\in H$ for all 
$1\le i\neq j\le r$.
Similarly, applying it to the group $\GL(n-r+1,R)$ embedded in 
$\GL(n,R)$ on the last $n-r+1$ rows and columns, we get that
$z_{ij}(ab,c),z_{ij}(ba,c)\in H$ for all $r+1\le i\neq j\le n$.
\par\smallskip
$\bullet$ Now if  $z_{st}(ab,c),z_{st}(ba,c)\in H$ for such an
extra position $(s,t)$, then we get inside $H$ the configuration 
of elementary commutators accounted for by Theorem D, and
can conclude that $H=[E(n,A),E(n,B)]$.
\par\smallskip
$\bullet$ On the other hand, let $y_{st}(a,b)\in H$, and let,
say $1\le s\neq t\le r$. Then applying Lemma 7 again, now to
$y_{st}(a,b),y_{sh}(a,b)\in H$, we can conclude that 
$z_{th}(an,c),z_{th}(ba,c)\in H$ and we are in the conditions
of the previous item. Again, we can conclude that 
$H=[E(n,A),E(n,B)]$.
\end{proof}

Of course, combining this result with Lemma 7, it is now easy 
to obtain a sharper form of Theorem 2 in similar spirit. The 
reason we failed to notice this refinement in \cite{NZ4} was
that we were expecting a statement in the style of Theorem C,
rather than the one in the style of Theorem 4.


\section{Multiple commutators}

To state the results of this section, we have to recall some further 
pieces of notation from \cite{yoga-1,Hazrat_Zhang_multiple,yoga-2, RNZ5, Hazrat_Vavilov_Zhang, Stepanov_universal, NZ3}.
Namely, let $H_1,\ldots,H_m\le G$ be subgroups of $G$. There are 
many ways to form a higher commutators of these
groups, depending on where we put the brackets. Thus, for three
subgroups $F,H,K\le G$ one can form two triple commutators
$[[F,H],K]$ and $[F,[H,K]]$. In the sequel, we denote by 
$\llbracket H_1,H_2,\ldots,H_m\rrbracket$ {\it any\/} higher 
mixed commutator of $H_1,\ldots,H_m$, with an arbitrary 
placement of brackets. Thus, for instance, $\llbracket F,H,K\rrbracket$ 
refers to any of the two arrangements above.
\par
Actually, the primary attribute of a bracket arrangement
that plays major role in our results is its cut point.
Namely, every higher commutator subgroup
$\llbracket H_1,H_2,\ldots,H_m\rrbracket$ can be uniquely written as
a double commutator
$$ \llbracket H_1,H_2,\ldots,H_m\rrbracket=
\Big[\llbracket H_1,\ldots,H_s\rrbracket,
\llbracket H_{s+1},\ldots,H_m\rrbracket\Big], $$
\noindent
for some $s=1,\ldots,m-1$. This $s$ is called the cut point
of our multiple commutator. 
\par
For {\it non-commutative\/} rings there is another aspect that 
affects the final answer. Namely, in this case symmetrised product 
of ideals is not associative. For instance, for three ideals 
$A,B,C\unlhd R$ one has
$$ (A\circ B)\circ C= ABC+BAC+CAB+CBA, $$
\noindent
whereas 
$$ A\circ (B\circ C)= ABC+ACB+BCA+CBA, $$
\noindent
that in general do not coincide. 
\par
To account for this, in the sequel 
we write $\llparenthesis I_1\circ\dots\circ I_m\rrparenthesis$
to denote the symmetrised product of $I_1,\ldots,I_m$ with 
an arbitrary placement of parenthesis. Thus, for instance,
$\llparenthesis A\circ B\circ C\rrparenthesis$ may refer either to
$(A\circ B)\circ C$, or to $A\circ (B\circ C)$, depending. In
the sequel the initial bracketing of higher commutators will be 
reflected in the parenthesizing of the corresponding multiple 
symmetrised products.
\par
In this notation \cite{NZ3}, Theorem 1, combined with 
\cite{yoga-1}, Theorem 8A, and \cite{Hazrat_Vavilov_Zhang}, 
Theorem 5A, can be stated as follows.
Similar results in the more general context of Bak's unitary 
groups\footnote{This is indeed a more general context, since
$\GL(n,R)$ are interpreted as a special case of unitary groups,
those over hyperbolic form rings.} were obtained in \cite{RNZ5, NZ6}.

\begin{oldtheorem}
Let $R$ be any associative ring with $1$, and either $n\ge 4$, 
or $n=3$ and $R$ is quasi-finite. Further, let
$I_i\unlhd R$, $i=1,\ldots,m$,  be two-sided ideals of $R$. 
Consider an arbitrary arrangement of brackets\/ $[\![\ldots]\!]$
with the cut point\/ $s$. Then one has
$$
\big\llbracket E(n,I_1),E(n,I_2),\ldots,E(n,I_m)\big\rrbracket=
\Big[E\big(n,\llparenthesis I_1\circ\ldots\circ I_s\rrparenthesis\big),
E\big(n,\llparenthesis I_{s+1}\circ\ldots\circ I_m\rrparenthesis
\big)\Big], $$
\noindent
where the  parenthesizing of symmetrised products on the right 
hand side coincides with the bracketing of the commutators on 
the left hand side.
\end{oldtheorem}

Of course, the first question that immediately occurs is whether 
the above quadruple and multiple elementary commutator formulas
also hold for $\GL(3,R)$ over arbitrary associative rings. This
question was already stated as \cite{NZ3}, Problem 1. The
results of \S~5 make it even more imperative. Though we
are rather more inclined to believe in the positive answer, up to 
now all our attempts to verify it by a direct calculation in the style 
of \cite{NZ3}, Lemma 7, failed.

\begin{Prob} 
Prove that Lemma $11$ and Theorem~{\rm E} hold also for $n=3$.
\end{Prob}

In conjunction with Theorem 1 the above Theorem C allows 
to give a very slick generating set for the multiple commutator 
subgroups
$\llbracket E(n,I_1),E(n,I_2),\ldots,E(n,I_m)\rrbracket$.

\begin{The}\label{t4}
Let $R$ be any associative ring with $1$, and either $n\ge 4$, 
or $n=3$ and $R$ is quasi-finite. Further, let
$I_i\unlhd R$, $i=1,\ldots,m$,  be two-sided ideals of $R$. 
Consider an arbitrary arrangement of brackets\/ 
$\llbracket\ldots\rrbracket$ with the cut point\/ $s$. Then the 
mixed multiple commutator subgroup
$$ \llbracket E(n,I_1),E(n,I_2),\ldots,E(n,I_m)\rrbracket $$
\noindent
is generated by the double elementary commutators
$$ [t_{ij}(a), t_{hk}(b)],\qquad 
a\in\llparenthesis I_1\circ\ldots\circ I_s\rrparenthesis,\quad
b\in \llparenthesis I_{s+1}\circ\ldots\circ I_m\rrparenthesis, $$
\noindent
where the parenthesizing of the above symmetrised products 
coincides with the bracketing of the commutators before and 
after the cut point, respectively.
\end{The}

Of course, these generating sets could be further reduced in 
the spirit of Theorems~D or 4.


\section{Final remarks}

Below we collect some of the most immediate open problems,
related to the results of the present paper.

Modulo $E(n,R,A\circ B)$ the elementary commutators
$y_{ij}(a,b)\in[E(n,A),E(n,B)]$ behave as symbols in algebraic
$K$-theory, and are subject to very nice relations, see
\cite{NZ1, NZ3, NZ4}. Even so, an explicit calculation of the
quotient 
$$ [E(n,A),E(n,B)]/E(n,R,A\circ B) $$ 
\noindent
turned out to be quite a challenge
and so far we succeeded in getting conclusive results only over
Dedekind rings. A similar question for the larger quotient
$$ [E(n,A),E(n,B)]/E(n,A\circ B) $$ 
\noindent
seems to be much harder.

\begin{Prob}
Find defining relations among the elementary commutators
$y_{ij}(a,b)\in[E(n,A),E(n,B)]$  modulo $E(n,A\circ B)$.
\end{Prob}

Theorem 1 suggests to resuscitate the approach to relative 
localisation developed in our joint papers with Roozbeh Hazrat
\cite{Hazrat_Zhang, RNZ1, RNZ2, Hazrat_Zhang_multiple, RNZ5,
Hazrat_Vavilov_Zhang}.
The heft of the technical difficulty in applying relative localisation 
to multiple commutator formulas 
\cite{Hazrat_Zhang_multiple, RNZ5, Hazrat_Vavilov_Zhang}
stemmed from the necessity to develop massive chunks of the 
conjugation calculus and commutator calculus \cite{yoga-1, yoga-2,
Hazrat_Vavilov_Zhang} for the new
type of generators listed in Theorem A. We believe that now, 
that we have Theorems B and 1, these calculations could be
reduced to a {\it fraction\/} of their initial length.

\begin{Prob}
Apply Theorem~$1$ to rethink localisation proofs of the
multiple commutator formulas.
\end{Prob}

As of today, the remarkable approach via universal localisation 
as developed by Alexei Stepanov \cite{Stepanov_nonabelian, Stepanov_universal} only works for algebraic groups over 
commutative rings. On the other hand, further applications of 
the initial versions of relative localisation were effectively 
blocked by technical obstacles.
\par
Such agreeable generating sets as found in Theorems 1 and 5 
could be especially advantageous for improving bounds in results 
on multiple commutator width, in the spirit of Alexei Stepanov, 
see \cite{SiSt, SV11, Stepanov_universal} and a survey in 
\cite{Porto-Cesareo}.

\begin{Prob}
Apply Theorems $1$ and $5$ to get results on multirelative
commutator width.
\end{Prob}

It would be natural to generalise results of the present paper
to more general contexts. The following development seems 
to be immediate.

\begin{Prob}
Generalise results of the present paper to Chevalley groups.
\end{Prob}

We believe that in this context all fragments of the necessary calculations were already elaborated: analogues of Lemma 5
by Alexei Stepanov in \cite{Stepanov_calculus, Stepanov_nonabelian},
analogues of Lemma 6 in our joint papers with Roozbeh Hazrat \cite{RNZ2, Hazrat_Vavilov_Zhang_generation}, and, finally, 
analogues of Lemma 7 in our recent works \cite{NZ2, NZ5}.
Of course, they were not stated there this way, but to 
establish them in the desired forms would only require some 
moderate processing of the existing proofs.

The following problem is similar, but seems to be somewhat
harder.

\begin{Prob}
Generalise results of the present paper to Bak's unitary groups.
\end{Prob}

In fact, here too large fragments of the necessary theory 
were already developed in our previous joint papers with 
Anthony Bak and Roozbeh Hazrat
\cite{BV3, RH, RNZ1, RNZ5, Hazrat_Vavilov_Zhang}, and
in our recent preprint \cite{NZ6}. But in this case even 
the analogues of Lemma 5 and Theorem C are lacking, and 
it would also take considerably more work to bring the results 
from the above papers to the required form.

We thank Roozbeh Hazrat and Alexei Stepanov for long-standing 
close cooperation on this type of problems over the last decades.



\begin{thebibliography}{30}

\bibitem{AS}
H.~Apte, A.~Stepanov, 
\emph{Local-global principle for congruence subgroups
of Chevalley groups,}
Cent. Eur. J. Math. {\bf 12} (2014), no.~6, 801--812.

\bibitem{Bak} 
 A.~Bak, 
\emph{Non-abelian $\K$-theory: The
nilpotent class of $\K_1$ and general stability},
$K$--Theory \textbf{4} (1991), 363--397.

\bibitem{BV3} A.~Bak, N.~Vavilov,  \emph{Structure of hyperbolic unitary
groups\/{\rm  I:} elementary subgroups.} {Algebra Colloquium},
\textbf{7} (2000), no.~2, 159--196.

\bibitem{Bass_stable} 
H.~Bass, \emph{$\K$-theory and stable algebra,}
Inst. Hautes \'Etudes Sci. Publ. Math.  (1964), no.~22, 5--60.

\bibitem{borvav} 
Z.~I.~Borewicz, N.~A.~Vavilov,  
\emph{The distribution of subgroups in the full linear group 
over a commutative ring}, 
Proc. Steklov Inst. Math. \textbf{3} (1985), 27--46.


\bibitem{RH} R.~Hazrat, \emph{Dimension theory and nonstable $K_1$ of
quadratic modules.} {K-Theory}, \textbf{27} (2002), 293--328.

\bibitem{yoga-1}
R.~Hazrat, A.~Stepanov, N.~Vavilov, Zuhong Zhang,
\emph{The yoga of commutators,} 
J.~Math.\ Sci. \textbf{179} (2011), no.~6, 662--678.

\bibitem{Porto-Cesareo}
R.~Hazrat, A.~Stepanov, N.~Vavilov, Zuhong Zhang,
\emph{Commutator width in Chevalley groups},
Note di Matematica \textbf{33} (2013), no.~1, 139--170.

\bibitem{yoga-2}
R.~Hazrat, A.~Stepanov, N.~Vavilov, Zuhong Zhang,
\emph{The yoga of commutators, further applications,} 
J.~Math.\ Sci. \textbf{200} (2014), no.~6, 742--768.

\bibitem{RN} R.~Hazrat, N.~Vavilov, \emph{Bak's work on the
$K$-theory of rings, with an appendix by Max Karoubi.}
{J. $K$-Theory}, {\bf 4} (2009), 1--65.

\bibitem{RNZ1} R.~Hazrat, N.~Vavilov, Zuhong Zhang, Relative unitary
commutator calculus and applications, {\it J. Algebra} {\bf 343}
(2011) 107--137.



\bibitem{RNZ2} R.~Hazrat, N.~Vavilov, Zuhong Zhang, Relative commutator
calculus in Chevalley groups, {\it J. Algebra} 
 {\bf 385} (2013), 262--293.

\bibitem{Hazrat_Vavilov_Zhang_generation} 
R.~Hazrat, N.~Vavilov, Zuhong Zhang
\emph{Generation of relative commutator subgroups in Chevalley groups,} 
Proc.\ Edinburgh Math.\ Soc., \textbf{59}, (2016), 393--410.

\bibitem{RNZ5} 
R.~Hazrat, N.~Vavilov, Zuhong Zhang, 
\emph{Multiple commutator formulas for unitary groups}. 
Israel J.~Math., {\bf 219} (2017), 287--330.

\bibitem{Hazrat_Vavilov_Zhang} 
R.~Hazrat, N.~Vavilov, Zuhong Zhang,
\emph{The commutators of classical groups,} 
J.~Math.\ Sci., \textbf{222} (2017), no.~4, 466--515.

\bibitem{Hazrat_Zhang} 
R. Hazrat, Zuhong Zhang,
\textit{Generalized commutator formula,} 
Commun.\ Algebra, \textbf{39} (2011), no.~4, 1441--1454.

\bibitem{Hazrat_Zhang_multiple} 
R. Hazrat, Zuhong Zhang,
\textit{Multiple commutator formula,} 
Israel J.~Math., \textbf{195} (2013), 481--505.

\bibitem{vdK-group} 
W.~van der Kallen,
\textit{A group structure on certain orbit sets of unimodular rows,}
J. Algebra \textbf{82} (1983), 363--397.

\bibitem{M2} A.~W.~Mason, 
\emph{On subgroup of $\GL(n,A)$ which are generated by commutators, {\rm II}}. 
J. reine angew. Math., {\bf 322} (1981), 118--135.

\bibitem{Mason_Stothers} 
A.~W.~Mason, W.~W.~Stothers, \emph{On subgroups of $\GL(n,A)$ 
which are generated by commutators,} 
Invent.\ Math., \textbf{23} (1974), 327--346.

\bibitem{SiSt}  
A.~Sivatski, A.~Stepanov, \emph{On the word length of
commutators in\/ $\GL_n(R)$,} 
{$K$-theory}, {\bf 17} (1999), 295--302.

\bibitem{Stepanov_calculus}
A.~Stepanov,  
\emph{Elementary calculus in Chevalley groups over rings,} 
J.~Prime Res.\ Math., \textbf{9} (2013), 79--95.

\bibitem{Stepanov_nonabelian}
A.~V.~Stepanov,  \emph{Non-abelian $\K$-theory for Chevalley groups over rings,} J.~Math.\ Sci., \textbf{209} (2015), no.~4, 645--656.

\bibitem{Stepanov_universal}
A.~Stepanov,  \emph{Structure of Chevalley groups over rings via universal localization},
J.~Algebra, \textbf{450} (2016), 522--548.

\bibitem{Stepanov_Vavilov_decomposition}
A.~Stepanov, N.~Vavilov,
\emph{Decomposition of transvections\/{\rm:} a theme with variations}, 
$\K$-Theory, \textbf{19} (2000), no.~2, 109--153. 

\bibitem{SV11} A.~Stepanov, N.~Vavilov, 
\emph{On the length of commutators in Chevalley groups,}
{Israel J. Math.} \textbf{185} (2011), 253--276.

\bibitem{Suslin}
A.~A.~Suslin, \emph{The structure of the special linear group over 
polynomial rings}, 
Math.\ USSR Izv., \textbf{11} (1977), no.~ 2,  235--253.

\bibitem{Vaserstein_normal} 
L.~N.~Vaserstein,  
\emph{On the normal subgroups of the\/ $\GL_n$ of a ring,} 
Algebraic $\K$-Theory, Evanston 1980,
Lecture Notes in Math., vol. 854, Springer,
 Berlin et al., 1981, pp. 454--465.

\bibitem{NV18} N.~Vavilov, 
\emph{Unrelativised standard commutator formula,}
J. Math. Sci., N. Y. \textbf{243} (2019), no.~4, 527--534.

\bibitem{NV19} N.~Vavilov, 
\emph{Commutators of congruence subgroups in the arithmetic case},
J. Math. Sci., N. Y.
{\bf 479} (2019), 5--22.

\bibitem{Vavilov_Stepanov_standard} 
N.~A.~Vavilov, A.~V.~Stepanov, 
\emph{Standard commutator formula,} Vestnik St.~Petersburg State
Univ., Ser. 1, \textbf{41} (2008), no.~1, 5--8.

\bibitem{Vavilov_Stepanov_revisited} 
N.~A.~Vavilov, A.~V.~Stepanov, 
\emph{Standard commutator formulae, revisited,} Vestnik 
St.~Petersburg State Univ., Ser.1, \textbf{43} (2010), no.~1, 12--17.
        
\bibitem{NZ1} N.~Vavilov, Zuhong Zhang, 
\emph{Commutators of relative and unrelative elementary groups,
revisited,} 
J.~Math.\ Sci.,  N.~Y. \textbf{485} (2019), 58--71.

\bibitem{NZ2} N.~Vavilov, Zuhong Zhang, 
\emph{Generation of relative commutator subgroups in Chevalley groups {\rm II},}
 Proc.\ Edinburgh Math. Soc., {\bf 63} (2020), no.~2, 497--511.

\bibitem{NZ3} N.~Vavilov, Zuhong Zhang, 
\emph{Multiple commutators of
elementary subgroups\/{\rm:} end of the line,}
Linear Algebra Applications, \textbf{599}
(2020), 1--17.

\bibitem{NZ4} N.~Vavilov, Zuhong Zhang, 
\emph{Inclusions among commutators of elementary subgroups,} 
J. Algebra,
\textbf{} (2020), 1--26.

\bibitem{NZ5} N.~Vavilov, Zuhong Zhang, 
\emph{Commutators of relative and unrelative
elementary subgroups in Chevalley groups,}
Proc.\ Edinburgh Math. Soc.,  \textbf{} (2020), 1--18.

\bibitem{NZ6} N.~Vavilov, Zuhong Zhang, 
{\it Commutators of relative and unrelative
elementary unitary groups,}
Rivista Math. Iberoamericana, \textbf{} (2020), 1--40.

\bibitem{HongYou} 
Hong You, 
\emph{On subgroups of Chevalley groups which are 
generated by commutators,} 
J.~Northeast Normal Univ., (1992), no.~2, 9--13.

\end{thebibliography}
\end{document}